\documentclass[11pt]{amsart}

\usepackage{fullpage,amssymb,amsmath,amsthm,amsfonts,graphicx, subfigure, color}

\usepackage{stmaryrd}
\usepackage{subfigure}
\usepackage{xparse}
\usepackage{tikz}
\usepackage{mathrsfs}
\usetikzlibrary{calc}
\usetikzlibrary{decorations.shapes, positioning}
\usepackage{float}

\newtheorem{theorem}{Theorem}[section]
\newtheorem{lemma}[theorem]{Lemma}
\newtheorem{proposition}[theorem]{Proposition}

\theoremstyle{definition}

\newtheorem{assumption}[theorem]{Assumption}

\newtheorem{remark}[theorem]{Remark}

\newcommand{\Def}{:=}
\newcommand{\R}{\mathbb{R}}

\renewcommand{\H}{\ensuremath{\mathbb{H}}}

\newcommand{\tr}{\operatorname{tr}}
\newcommand{\tendsto}{\to}

\newcommand{\expect}{\mathbb{E}}
\newcommand{\Ex}{\mathbb{E}}
\newcommand{\Exp}{\mathbb{E}}

\newcommand{\prob}{\mathbb{P}}

\renewcommand{\Pr}{\prob}
\DeclareDocumentCommand \one { o }
{%
\IfNoValueTF {#1}
{\mathbf{1}  }
{\mathbf{1}\left\{ {#1} \right\} }%
}

\newcommand{\Var}{\operatorname{Var}}
\newcommand{\Cov}{\operatorname{Cov}}

\newcommand{\weakto}{\Rightarrow}

\newcommand*\nare[2]{\ensuremath{N^e_{#1}({#2})}}
\newcommand*\naro[2]{\ensuremath{N^o_{#1}({#2})}}

\newcommand*\WEV{\ensuremath{\atomvar}_{1,1}}
\newcommand*\WCM{\ensuremath{\operatorname{WCov}}}
\newcommand*\WEM{\ensuremath{\operatorname{WE}}}

\newcommand*\CLS[2]{\ensuremath{X^{\Wishartmatrix_{{#2}}}_{ {#1} }}}
\newcommand*\CLSf[2]{\ensuremath{X^{\Wmf{{#2}}}_{ {#1} }}}
\newcommand*\PCLSg[1]{\ensuremath{X^{}_{ {#1} }}}
\newcommand*\PCLS[2]{\ensuremath{X^{}_{ {#1}\otimes{#2} }}}

\newcommand*\cmtab[1]{\ensuremath{X_{#1}}}
\newcommand*\cmti[2]{\ensuremath{X^{W_{{#2}}}_{x^{#1}}}}
\newcommand{\Wishartmatrix}{\ensuremath{W}}
\newcommand{\Wm}{\Wishartmatrix}
\newcommand{\ambientmatrix}{\ensuremath{G}}
\newcommand{\submatrix}{\ensuremath{B}}
\newcommand{\smf}[1]{\mathbf{B}({#1})}
\newcommand{\Wmf}[1]{\Wm({#1})}
\newcommand{\atomvar}{\ensuremath{Z}}

\newcommand{\GFFp}{\ensuremath{\mathcal{G}^{\Omega}}}

\DeclareDocumentCommand{\FSS}{ O{s} }{ \mathscr{H}_{{#1}} }
\DeclareDocumentCommand{\FSN}{ O{\cdot} O{s} }{ \left\|{#1}\right\|_{{#2}} }
\DeclareDocumentCommand{\FSNr}{ O{\cdot} O{s} O{\rho} }{ \left\|{#1}\right\|_{{#2},{#3}} }
\DeclareDocumentCommand{\opnorm}{ O{\cdot} }{ \left\|{#1}\right\|_{\text{op}} }

\newcommand{\walkset}[1]{\mathcal{S}_{ {#1} }}
\newcommand{\walksetf}[2]{\mathcal{S}_{{#1}}({#2})}
\newcommand{\atomword}[1]{Z[ {#1} ]}
\newcommand{\rhoweight}[2]{\alpha_{{#1}}[ {#2} ]}

\newcommand{\DBLoint}{ \mathop{ \oint \hspace{-0.15cm} \oint}}

\begin{document}

\title{Spectra of Overlapping Wishart Matrices and the Gaussian Free Field}

\author{Ioana Dumitriu}
\address{Department of Mathematics, University of Washington}
\email{dumitriu@math.washington.edu}
\author{Elliot Paquette}
\address{Department of Mathematics, The Ohio State University}
\email{paquette.30@osu.edu}
\thanks{We acknowledge the support of the NSF through CAREER award DMS-0847661.
EP acknowledges the support of NSF Postdoctoral Fellowship DMS-1304057.}
\date{\today}

%    Abstract is required.
\begin{abstract}
Consider a doubly-infinite array of i.i.d.\ centered variables with
moment conditions, from which one can extract a finite number of
rectangular, overlapping submatrices, and form the corresponding
Wishart matrices. We show that under basic smoothness assumptions, 
centered linear eigenstatistics of such matrices converge jointly to a
Gaussian vector with an interesting covariance structure. This
structure, which is similar to those appearing in \cite{Borodin},
\cite{BorodinGorin}, and \cite{JohnsonPal}, can be described in terms of the height function, and leads
to a connection with the Gaussian Free Field on the upper
half-plane. Finally, we generalize our results from univariate polynomials to a special
class of planar functions. 
\end{abstract}

\maketitle

\section{Introduction}
\label{sec:intro}
%Wishart matrices are likely the oldest family of random matrices.  Originally studied for their relevance in multivariate statistics, they continue to be a fundamental object of study in random matrix theory.  While many of their properties are interesting from either the point of view of random matrix theory or of multivariate statistics, we will focus on the fluctuations of their linear statistics.

% \begin{itemize}
% \item random Wishart matrices. Fix the assumptions, explain that 4th
%   moment matters as we will see. Mention universality. 
% \item linear statistics \begin{itemize}
%     \item ESD convergence (semicircle, Marcenko-Pastur). Zero-d.
%     \item Gaussian fluctuations (Jonsson, Anderson-Zeitouni,
%       Johansson, D.-Edelman; other methods (stochastic calc) Guionnet,
%       Cabanal-Duvillard?) One-d.
%     \item Can one obtain two-d? Overlapping matrices. Minor
%       processes. The Height. Borodin (Johnson-Pal, Borodin-Gorin). Our results;
%       why planar functions, first time beyond polynomials. General
%       form of the theorem.
% \end{itemize}
% \item Approach and Scherbina's results. 
% \item Structure of paper: \begin{itemize}
% \item Covariance calculation for moments; mention importance of 4th
%   moment \& correction term.
% \item Polynomial CLT.
% \item Extension to planar functions.
% \end{itemize}
% \end{itemize}

Alongside Wigner matrices, Wishart matrices are the oldest family of
random matrices, tracing their name to a 1928
\emph{Biometrika} paper by the statistician John Wishart \cite{Wishart}. 
In the centered version, Wishart matrices take the form $W = \Sigma^{1/2} X
X^{T} \Sigma^{1/2} $, where $X$ is a $n \times m$ matrix of i.i.d.\ variables
(often standard normals) with mean $0$ and variance $1$, and $\Sigma$
is the $m \times m$ positive-definite covariance matrix. For the purposes of this
paper, we will only consider the ``null'' case, i.e. when $\Sigma = I_m$ (the
identity matrix), but we will not make normality assumptions on the
entries of the matrix $X$. 

One of the main use of Wishart matrices has been to model sample covariance matrices; as such, their eigenvalue and eigenvector statistics (for simplicity we will refer to them as \emph{eigenstatistics}) have been used to devise likelihood estimation models, and hence have been the object of study by a large community and from different perspectives. In the last couple of decades, interest has been particularly high in the fact that many of these eigenstatistics
exhibit universal behavior, i.e. their asymptotics do not depend on
the distribution of the entries of $X$ (starting with universality of the fluctuation of the largest eigenvalue, \cite{sosh_univ_wish}, \cite{peche_univ_wish}, \cite{bao_pan_zhou_13}, central limit theorems for linear statistics \cite{BaiSilverstein}, to full universality away from the edge \cite{tao_vu_univ_wish}).

The simplest and most widely-known of these eigenstatistics is the empirical eigenvalue distribution, with density given by
\[
\sigma_{W_n} = \frac{1}{n} \sum_{i=1}^n \delta_{\lambda_i}~,
\]
where $\lambda_1, \ldots, \lambda_n$ are the $n$ eigenvalues of the matrix $W$. If one scales down the eigenvalues by $n$, the scaled version of the empirical spectral distribution converges in probability to the well-known Mar\v{c}enko-Pastur distribution $\sigma_{MP}$ (\cite{marcenko67a}, \cite{Jonsson}). One can interpret this as follows: a sequence of random distributions, corresponding to densities $\sigma_{W_n}$, converges to a fixed distribution, $\sigma_{MP}$; this defines a sort of zero-dimensional process. 

This convergence phenomenon can be analyzed and refined, by examining the fluctuations from $\sigma_{MP}$. To this end, for a function $f$, one may define the \emph{(centered) linear statistic $f$ of a matrix $W$}:
$$\cmtab{f}^{\Wm}:=\tr( f(\Wm)) - \Ex \tr (f (\Wm)) =
\sum_{i=1}^m f(\lambda_i) - \Ex f(\lambda_i)~,$$ where $\lambda_1,
\ldots, \lambda_m$ are the eigenvalues of $\Wm$ and $\Ex$ denotes expectation over the ensemble. Aside from trying to understand a deep and important mathematical phenomenon, the asymptotic properties of centered linear statistics for sample covariance matrices have interesting potential applications \cite{RMSE}.

The centered linear statistics of Wishart matrices have been shown to have Gaussian fluctuations (\cite{Jonsson},  \cite{CabanalDuvillard}, \cite{BaiSilverstein}, \cite{anderson-zeitouni_clt_wish}, \cite{Shcherbina11}). More specifically, they converge to a Gaussian process on the real line; that is, for sufficiently smooth test functions $f,$ the linear statistics converge (without rescaling) to normal variables with computable covariance structure\footnote{Although not specifically for Wishart ones, this covariance structure has been shown in some cases to be universal, provided the limiting support of the eigenvalues is an interval, e.g. \cite{Johansson}.}. Such a Gaussian process can be seen as a one-dimensional process on the real line. 

It is therefore a natural question to ask if this deep phenomenon can be extended further, and if by digging a bit more one might uncover a two-dimensional process as a limit. This seems to be indeed the case, as we shall show here; toward this purpose one must add one more dimension, and focus on ``overlapping'' matrices. Investigations of the joint eigenvalue distributions (or statistics thereof) in overlapping matrices have sometimes been called ``minor processes'' (\cite{johansson-nordenstam06}, \cite{forrester_nordenstam09}).

This question has been shown to have affirmative answer in the context of Wigner ensembles \cite{Borodin}, $\beta$-Jacobi ensembles \cite{BorodinGorin}, and $d$-regular graphs \cite{JohnsonPal}. 
% As with many linear statistics of classical random matrix models,
% linear statistics of random matrices are known to satisfy a central
% limit theorem for many classes of test functions $f$ (\cite{Johansson}
% for \cite{BaiSilverstein}.  In particular, for sufficiently smooth test functions $f,$ the linear statistics converge, without rescaling, to normal variables.  
% The covariance structure of these normal variables has been observed
% to be universal, provided the limiting support of the eigenvalues is an
% interval~\cite{Johansson}. \textcolor{red}{you mean Wigner?}
More specifically, the covariance of these linear statistics was observed by
Borodin to be expressible in terms of the $0$-boundary Gaussian free field~\cite{Borodin}. He showed that the linear statistics of submatrices of a single Wigner behave like path integrals of test functions against a family of correlated Gaussian free fields (i.e., two-dimensional objects).  
In this vein, our objective here is two-fold: we will describe the asymptotics of linear statistics of eigenvalues of Wishart matrices formed from picking (overlapping) submatrices of a single large matrix, and we will show convergence of a large class of ``planar'' centered linear statistics to spatial averages of the (two-dimensional) Gaussian free field.

To this end, we define an infinite array of random variables $\atomvar_{i,j},$ $i \geq 1$ and $j \geq 1.$  
%These entries are centered and have variance $1$ in the real case or variance QUOI in the complex case.  
%\subsection{Assumptions}
We allow for $\atomvar_{i,j}$ to be an independent family of real, complex or quaternion random variable, corresponding to the classical $\beta=1,2,4$ trichotomy.  In all cases, all components (e.g. $\Re \atomvar_{i,j}$ and $\Im \atomvar_{i,j}$) of the variable are mutually independent.  Assume further that the variables satisfy the following moment assumptions.
%\begin{assumption}
%	\label{a:all}
%\begin{eqnarray*} % \label{conds}
%\Ex \left| \atomvar_{i,j} \right|^2 = 1,
%\hspace{1cm}
%\Ex \left| \atomvar_{i,j} \right|^4 = 1 + \frac{2}{\beta},
%\text{ and} 
%\hspace{1cm}
%\sup_{i,j} \Ex \left| \atomvar_{i,j} \right|^k < \infty, k \in \mathbb{N}.
%\end{eqnarray*}
%\end{assumption}
%
\begin{assumption}
	\label{a:4e}
	There is an $\epsilon > 0$ so that
\begin{eqnarray*}
\Ex \atomvar_{i,j} = 0,
\hspace{1cm}
\Ex \left| \atomvar_{i,j} \right|^2 = 1,
\hspace{1cm}
\Ex \left| \atomvar_{i,j} \right|^4 = 1 + \frac{2}{\beta},
\text{ and} 
\hspace{1cm}
\sup_{i,j} \Ex \left| \atomvar_{i,j} \right|^{4+\epsilon} < \infty.
\end{eqnarray*}
\end{assumption}
\noindent The first, second and fourth moments are chosen to agree with the standard real, complex and quaternion Gaussian variables.
%\begin{remark}
%  By a standard truncation argument, it is possible to relax the assumptions on moments higher than $4$ to $\sup_{i,j} \Exp |Z_{i,j}|^{4+\epsilon} < \infty$ or even to a Lindeberg type condition on the $4$-th moment (see e.g.~\cite{Shcherbina11}). The $4$th moment condition is necessary for obtaining the Gaussian free field as a limit, as we will show in Section \ref{sec:covar}. 
%\end{remark}

From this infinite matrix, we extract rectangular submatrices $\submatrix_i$ for $i=1,\ldots,k$ with sizes $m_i(L) \times n_i(L)$ that overlap on $m_{i,j}(L)$ rows and $n_{i,j}(L)$ columns. 
%Thus, each represents a matrix of noise where some amount of noise is shared between the various matrices.
From these submatrices, we form the Wishart matrices $\Wishartmatrix_i \Def \submatrix_i^*\submatrix_i/L$ for $i \in [k],$ and we study the limiting behavior of $(\cmti{p_1}{1},\cmti{p_2}{2},\ldots,\cmti{p_k}{k})$ for natural numbers $p_i$ provided that all of intersection parameters scale linearly with $L.$  

Our first result explores the distributional limits of these centered linear statistics, which can be viewed as a generalization of known distributional limit theorems for centered linear statistics of a single matrix (i.e., when $W_1 = W_2$).
\begin{proposition}
\label{prop:CLS}
Suppose that $\left\{ \atomvar_{i,j} \right\}$ satisfy Assumption~\ref{a:4e}, that $\frac{m_i}{L} \to \mu_i,$ $\frac{n_i}{L} \to \nu_j$ and that 
\(
\frac
{m_{i,j}n_{i,j}L^2}
{m_in_im_jn_j}
\to \theta_{i,j},
\)
for all $1 \leq i,j \leq k.$  It follows that $(\cmti{p_1}{1},\cmti{p_2}{2},\ldots,\cmti{p_k}{k})$ converges in distribution to a centered Gaussian vector $(\xi_1,\xi_2,\ldots,\xi_k)$ with covariance
\[
\Ex \xi_{i} \xi_j
=-\frac{4 p_i\cdot p_j}{\beta \pi^2} 
\!\!
\DBLoint
\!\!
\left(\mu_i \!+\! \nu_i \!+\! 2\Re \zeta_i\right)^{p_i-1}
\!\!
\left(\mu_j \!+\! \nu_j \!+\! 2\Re \zeta_j\right)^{p_j-1}
\!\!
K_{\theta_{i,j}}(\zeta_i,\zeta_j)
\frac{\Im \zeta_i}{\zeta_i}
\frac{\Im \zeta_j}{\zeta_j}
d\zeta_i\,d\zeta_j,
\]
where the contours are semicircles in $\H$ centered at the origin with radii $\sqrt{\mu_i\nu_i}$, respectively, $\sqrt{\mu_j\nu_j}$, and where
\(
K_{\theta_{i,j}}(\zeta_i,\zeta_j)
=
\log \left|
\frac{
  {\theta_{i,j}^{-1} - \zeta_i \zeta_j}
}
{
  {\theta_{i,j}^{-1} - \zeta_i \bar{\zeta_j}}
}
\right|
\)
.
\end{proposition} 
\begin{remark}
	In the case that the $4$-th moments of $\left\{ \atomvar_{i,j} \right\}$ are not equal to $1 + \frac{2}{\beta},$ but are instead all equal to some other constant, the same proposition holds with an added term to the covariance (given by $\WEM_{i,j}$ from Lemma \ref{lem:wish_cov_GF}, when $\beta=1$).
\end{remark}
Similar covariance structures appear in the work of Borodin on minors of Wigner matrices~\cite{Borodin}, of Borodin and Gorin in the case of $\beta$-Jacobi matrices~\cite{BorodinGorin}, and of Johnson and Pal in the case of adjacency matrices of regular graphs \cite{JohnsonPal}.  Further, the exact limiting process has a description in terms of the same correlated Gaussian free fields.  We will explain how this theorem can be reformulated in terms of the distributional convergence of the \emph{height function}.  In this way, this proposition can be seen as an analogue of what Borodin proves for Wigner matrices (Theorem 2 of~\cite{Borodin}).  

There are two key ingredients in this approach: the first one, as mentioned, is establishing and explaining the connection to the Gaussian free field via the height function. The second is the introduction of planar test functions. We present the connection in Section \ref{height}, we give the necessary definitions and our main result in Section \ref{planarf}, and we sketch the outline of the proof and give the ``main ingredients'' (supporting results) in Section \ref{outlinepf}. The rest of the paper is then dedicated to proving the main ingredients.

\subsection{The Height Function} \label{height}

%We will show the connection of the height function to the pullback of the zero-boundary Gaussian Free Field on the upper half plane.  

We will give a reformulation of this proposition in the case where all matrices come from the upper left corner of a matrix $G$ of size $\left[ \mu L\right] \times \left[ \nu L \right]$ with entries $Z_{i,j}.$  This will provide the connection to the Gaussian free field.  For any $y \geq 0$ define $\smf{y}$ as having entries
\(
\smf{y}(i,j) = Z_{i,j},
\)
for $1 \leq i \leq \left[ y \mu L \right]$ and $1 \leq j \leq \left[ y \nu L \right].$  Symmetrize $\smf{y}$ by taking $\Wmf{y} = \smf{y}^* \smf{y}/L.$  As we will only consider the fluctuations of these matrices, the order of the symmetrization is inconsequential; hence we may assume $\mu \geq \nu,$ as in the other case the distribution of the centered linear statistics are the same. 

For an interval $I \subseteq \mathbb{R}$, let $N_I^{\Wmf{y}}$ be the number of eigenvalues of $\Wmf{y}$ in $I.$  Define the scaled centered height function to be the integer-valued function on $\mathbb{R} \times [0,1]$ given by
\[
H(x,y) = N_{[x,\infty)}^{\Wmf{y}} - \Exp N_{[x,\infty)}^{\Wmf{y}}.
\]
As we have that $H(x,y) \to 0$ almost surely as $|x| \to \infty,$ we have that for any $f \in C_c^\infty(\R)$, 
\[
-\int_{\R} f'(x) H(x,y)~dx = \CLSf{f}{y}~,
\]
using integration by parts.

From the Mar\v{c}enko-Pastur law, we expect that all increases of the height function occur for $x$ and $y$ in $\mathbb{R} \times [0,1]$ such that $\left(\sqrt{\mu} - \sqrt{\nu}\right)^2 \leq \frac{x}{y} \leq \left( \sqrt{\mu} + \sqrt{\nu} \right)^2$.  We will define a coordinate chart of this region of the plane to map it correctly to the upper half plane.  
\begin{figure}
%\begin{minipage}{0.9\linewidth}
  \begin{tikzpicture}[scale=0.80]
      \coordinate (a) at (0.33,4);
      \coordinate (b) at (3,4);

      \fill[color = red] (a) circle (0.1cm);
      \draw (a) node[above=0.3] { $ \left( \sqrt{\mu} - \sqrt{\nu}\right)^2 $};
      \fill[color = red] (b) circle (0.1cm);
      \draw (b) node[above right=0.3 and 0.0] { $ \left( \sqrt{\mu} + \sqrt{\nu}\right)^2 $};

    \draw[thin,dotted] (0,0) grid (4.1,4.1);
    \draw[->] (0,0) -- (4.2,0) node[right] { $x$ };
    \draw[->] (0,0) -- (0,4.2) node[above] { };

      \draw (-1pt, 4) -- (1pt, 4) node[left] {$y=1$};
      \draw (-1pt, 3) -- (1pt, 3) node[left] {$\frac{3}{4}$};
      \draw (-1pt, 2) -- (1pt, 2) node[left] {$\frac{1}{2}$};
      \draw (-1pt, 1) -- (1pt, 1) node[left] {$\frac{1}{4}$};

    \draw[-,dashed] (0,0) -- (a) -- (b) -- cycle;

%% Trajectories on the left
    \draw[-] ( $0.75*(a)$ ) --( $0.75*(b)$ ); 
    \draw[thick,->] ( $0.75*(b)$ ) --( $0.5625*(a)+0.1875*(b)$ ); 
    \draw[-] ( $0.5*(a)$ ) --( $0.5*(b)$ ); 
    \draw[thick,->] ( $0.5*(b)$ ) --( $0.375*(a)+0.125*(b)$ ); 
    \draw[-] ( $0.25*(a)$ ) --( $0.25*(b)$ ); 
    \draw[thick,->] ( $0.25*(b)$ ) --( $0.1875*(a)+0.0625*(b)$ ); 
    \begin{scope}[shift={ (5.75,0)}]
%% Omega symbol
      \draw[thick,->] (0,2) arc (70 : 110 : 2 );
      \draw[thick,-] (0,2) arc (70 : 90 : 2 )
      node[align=center,above] {$\Omega^{-1}$};
     \end{scope}

    \begin{scope}[shift={ (10,0)}]
    \draw[thin,dotted] (-3.9,0) grid (3.9,3.9);
    \draw[-,dashed] (4,0) arc (0: 180: 4) -- cycle;    
    \draw[->] (0,0) -- (0,4.2) node[above] { };
    \coordinate (O2) at (0,0);

%% Trajectories on the right
%    \draw[-] (2,0) arc (0 : 180 : 2 );
%    \draw[very thick,->] (2,0) arc (0 : 120 : 2 );
%    \draw[-] (2.82,0) arc (0 : 180 : 2.82 );
%    \draw[very thick,->] (2.82,0) arc (0 : 120 : 2.82 );
%    \draw[-] (3.46,0) arc (0 : 180 : 3.46 );
%    \draw[very thick,->] (3.46,0) arc (0 : 120 : 3.46 );

%      \draw (4,-1pt) -- (4,1pt) node[below] {$1$};
%      \draw (3.46,-1pt) -- (3.46,1pt) node[below] {$\sqrt{\frac{3}{4}}$};
%      \draw (2.82,-1pt) -- (2.82,1pt) node[below] {$\sqrt{\frac{1}{2}}$};
%      \draw (2,-1pt) -- (2,1pt) node[below] {$\sqrt{\frac{1}{4}}$};
    \draw[-] (1,0) arc (0 : 180 : 1 );
    \draw[very thick,->] (1,0) arc (0 : 120 : 1 );
    \draw[-] (2,0) arc (0 : 180 : 2 );
    \draw[very thick,->] (2,0) arc (0 : 120 : 2 );
    \draw[-] (3,0) arc (0 : 180 : 3 );
    \draw[very thick,->] (3,0) arc (0 : 120 : 3 );

      \draw (4,-1pt) -- (4,1pt) node[below] {$1$};
      \draw (3,-1pt) -- (3,1pt) node[below] {${\frac{3}{4}}$};
      \draw (2,-1pt) -- (2,1pt) node[below] {${\frac{1}{2}}$};
      \draw (1,-1pt) -- (1,1pt) node[below] {${\frac{1}{4}}$};

     \end{scope}
\end{tikzpicture}
%\includegraphics[height=8in]{Omega.pdf}
%\hspace{-10cm}
\caption{The bijection $\Omega^{-1}$} \label{fig:chart}
%\end{minipage}
\end{figure}
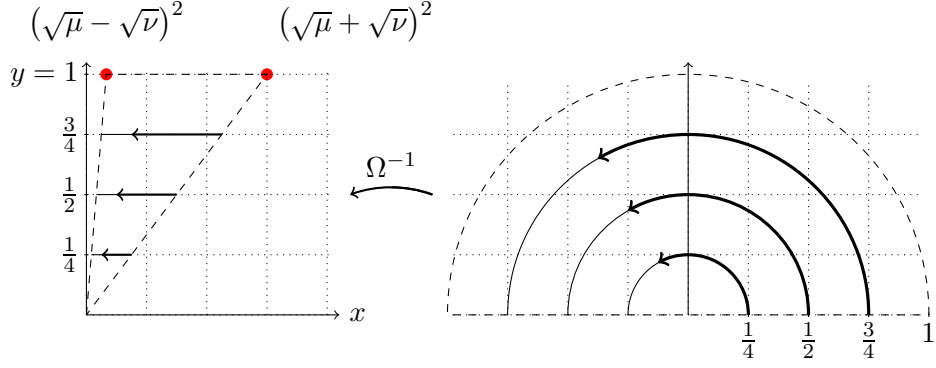
Let $\Omega^{-1}(z)~:\H \to \H$ be given by 
\begin{align*}
\Omega^{-1}(z) &\Def (x(z),y(z)) \\
x(z) &= |z|\left(\mu + \nu\right) + 2\sqrt{\mu\nu}{\Re z} \\
y(z) &= |z|.
\end{align*}
This can be seen to bijectively map concentric semicircles centered at $0$ to horizontal line segments in the support of the Mar\v{c}enko-Pastur law (see Figure~\ref{fig:chart}). We can now reinterpret Proposition~\ref{prop:CLS} in terms of convergence of $H$ to the pullback $\GFFp$ of the Gaussian free field under $\Omega.$  This can be defined as a random Schwartz distribution $\GFFp$ acting on $C_c^\infty(\H),$ for which all
\(
\int f(z) \GFFp(z) |dz|^2 \Def \left< f, \GFFp \right>
\) 
are normally distributed and 
\[
\Cov( \left< f, \GFFp \right>, \left< g, \GFFp \right> ) = \int_{\H^2} f(z_1) g(z_2) K_{\Omega}(z_1,z_2) |d\Omega(z_1)|^2 |d\Omega(z_2)|^2,
\]
where 
\(
K_\Omega(z_1,z_2)=
\frac{1}{2\pi} 
\log \left| \frac
{\Omega(z_1)- \bar \Omega(z_2)}
{\Omega(z_1)- \Omega(z_2)}
\right|
\) is the Green's function for the upper half plane composed with the coordinate chart $\Omega$.  The Gaussian free field, while being too rough to be a function, is sufficiently regular that it is also possible to define its action on rougher test functions.  Especially, we can define
\[
\int_{
\H^2
} f(x,y) \GFFp(x,y)\,dxd\rho(y)
\]
for probability measures $\rho$ and we can formally recast Proposition~\ref{prop:CLS} as showing a form of finite-dimensional marginal convergence of $H$ to $\GFFp(z).$  Namely, we show in this paper  that for some class of test functions $f$ (to be defined later),  
\begin{equation}
\label{eq:desired}
\int_{\H^2} f(x,y) H(x,y)\,dxd\rho(y)
\weakto
\sqrt{
  \frac{c}
{
  {\beta}
}
}
\int_{\H^2} f(x,y) \GFFp(x,y)\,dxd\rho(y)
\end{equation}
for some absolute constant $c>0$ 
as $L\to \infty.$
%\begin{equation}
%\label{eq:desired}
%\int_{\substack{|z|=y \\
%z \geq 0
%}} f(z) H(z) |dz|
%\weakto
%\frac{1}
%{
%\sqrt{
%{2\beta \pi}
%}}
%\int_{\substack{|z|=y \\
%z \geq 0
%}} f(z) \GFFp(z) |dz|.
%\end{equation}

\begin{remark}
%{\bf No pullback GFF comment!}
It may be tempting to conclude that $H \circ \Omega^{-1}$ converges to the standard Gaussian free field.
This does not follow in a natural sense from Proposition~\ref{prop:CLS}.
The difficulty lies in the behavior of the derivative of $\Omega^{-1}(z),$ which is singular along the real axis. 
%If we parameterize the first integral \textcolor{red}{UM, WHAT??}, then we have
If we integrate $H \circ \Omega^{-1}$ against a test function in $\H^2,$ along a semicircle of radius $y,$ we get
\[
\int_{0}^\pi f(ye^{i\theta}) (H \circ \Omega^{-1})(ye^{i\theta}) yd\theta
=
\int_{\theta=0}^{\theta=\pi} (f \circ \Omega)(x,y) H(x,y) \frac{dx}{2\sqrt{\mu \nu}\sin \theta}.
\]
where we have changed variables so that
$\cos \theta = \frac{x-(\mu + \nu)y}{2y\sqrt{\mu\nu}}.$  To apply Proposition~\ref{prop:CLS}, we need $(f \circ \Omega)(x,y) / \sin \theta$ to be a polynomial in $x$, which puts an awkward set of boundary conditions on the class of functions for which we can prove weak convergence in probability.
%$f$ to vanish near the
%However, composing with $\Omega^{-1}$ leads to singularities in the derivative of $f$ near the real axis, which causes technical complications.
%On the other hand, from Proposition~\ref{prop:CLS}, it does follow immediately that for fixed $y$ and polynomial $f,$ 
%\[
%\int_{\R} f(x,y) H(x,y) dx \weakto 
%\sqrt{
%\frac{2\mu\nu}
%{\beta \pi}
%}\int_{|z|=y} (f\circ\Omega^{-1})(z) \GFF(z) \sin \theta |dz|,  
%\]
%where $\sin \theta = \frac{\Im z}{|z|}$ (CE: constants are probably off).  Note that there are subtle technical differences between these statements, such as the truncation of the integral and the location of the extra factor of $\sin \theta.$  
\end{remark}

\subsection{Planar test functions} \label{planarf}

%Before proceeding to how to resolve these technical issues, we first generalize the setup.  
The test functions in~\eqref{eq:desired} are to a certain extent unsatisfactory: they are of the form $f(x) \delta(r-y)$ for $\delta$ the standard Dirac delta function; they have a planar definition, but are still one-variable objects. 
As mentioned before, the Gaussian free field is a two-dimensional object; we would like to apply it to \emph{bona fide} two-dimensional functions. 

We take a step in this direction; note that this is the first time when a class of (actual) two-dimensional functions appears in a Gaussian free field convergence theorem in this context; previous results (\cite{Borodin}, \cite{BorodinGorin}, \cite{JohnsonPal}) involved exclusively univariate polynomial functions multiplied by $\delta(r-y)$. 

Let $\psi(x,y)$ be a two-dimensional function. We define the centered linear statistic $\PCLSg{\psi}$ by
\[
\PCLSg{\psi} \Def \int_{0}^\infty \CLSf{\psi(\cdot,y)}{y}~dy~,
\] 
for those functions $\psi$ for which this makes sense.  Note that $\Wmf{y}$ is a step function, and hence this integral is a finite weighted sum of linear statistics. 

%This formalization does not allow us to recover the single-matrix linear statistics
%, which would be desirable for an honest generalization.  
%Thus, we would like to allow for a class of planar measures.  

We need to consider measures which are sufficiently smooth in the $x$ coordinate to have a limit in this scaling, and thus we choose to examine test functions of the form $\psi(x,y)\rho(y)$ where $\psi$ is a Borel-measurable function and $\rho$ is a compactly supported measure.  For such measures, the linear statistic $\PCLSg{\psi(x,y)\rho(y)}$ is naturally given by
\[
\PCLSg{\psi(x,y)\rho(y)} = \int_{0}^\infty \CLSf{\psi(\cdot,y)}{y}~d\rho(y)~,
\]
which can be seen to agree with the definition given for $X_{\psi}$ when $\rho$ is absolutely continuous.

The precise regularity in the $x$ coordinate we require is in terms of the fractional Sobolev norms.   Define the norm $\FSN$ by
\[
\FSN[\phi]^2 = \int_{\R} (1+k^2)^{s} | \hat \phi(k) |^2~dk,
\] 
with $\hat\phi$ the Fourier transform
\[
\hat\phi(k) = \frac{1}{2\pi} \int_{\R} e^{-ikx} \phi(x)~dx.
\]
For a fixed, compactly supported probability measure $\rho$, we define the norm
\[
\FSNr[\psi(x,y)]^2 \Def \int_{0}^\infty \FSN[\psi(\cdot,y)]^2~d\rho(y).
\]
We may now state our main theorem for planar linear statistics.
\begin{theorem}
\label{thm:pcls}
Suppose Assumption~\ref{a:4e} on $\left\{ \atomvar_{i,j} \right\}$ holds,
and suppose $\beta=1.$  For any $\psi$ and any compactly supported probability measure $\rho$ for which $\FSNr[\partial_x\psi] < \infty$ for some $s > \tfrac 32,$ 
the centered linear statistic
\[
\PCLSg{\partial_x \psi(x,y) \rho(y)}
=
\int_{\H} \psi(x,y)H(x,y) dx d\rho(y) 
\]
%\PCLSg{\psi\rho}
converges in distribution to a centered normal $\xi$ with variance given by
\(
\Var \xi = \frac{8\mu \nu}{\beta \pi} \Var \left< \psi\rho, \GFFp \right>
\)
as $L \to \infty.$
\end{theorem}

\begin{remark}
  \label{rem:otherbeta}
  The proof given here works as well for $\beta=2$ or $4,$ however we rely on a variance bound of~\cite{Shcherbina11} which is explicitly proven only for $\beta=1$ (see \cite{Shcherbina11}, Proposition \ref{prop:shcher_var}). Notably, Shcherbina's proof is not hinging on the fact that $\beta=1$; although not explicitly written down, it can be verified with a little work that it is also applicable to $\beta = 2,4$
. Thus it can be used in conjunction with our proof to yield results similar to Theorem \ref{thm:pcls} for complex or quaternion entries.
\end{remark}

%\vspace{.5cm}

\subsection{Outline of proof for Theorem \ref{thm:pcls}} \label{outlinepf}

Our approach to this theorem is by a density argument.  We begin by showing the statement for a large class of polynomial test functions.

The first step was Proposition \ref{prop:CLS}; we then extend the results to planar functions of the type discussed in the previous section, as follows. Let $\rho$ be a probability measure with compact support on $[0,\infty)$ and for a real valued function $f,$ define the function $f \otimes \rho$ by $(f \otimes \rho)(x,y) = f(x)\rho(y).$  Hence the linear statistic $\PCLS{f}{\rho}$ is given by
\begin{align}
\label{eq:pcls}
\PCLS{f}{\rho} = \int_{0}^\infty \CLSf{f}{y} d \rho(y).
\end{align}
We ignore the trivial case where $\rho = \delta_0.$  In terms of this, we prove a version of the finite dimensional marginal convergence, and a generalization of Proposition \ref{prop:CLS}.  
\begin{proposition}
\label{prop:CLSplanar}
Suppose Assumption~\ref{a:4e} on $\left\{ \atomvar_{i,j} \right\}$ holds.
Fix polynomials $p_1,p_2,\ldots,p_k$ and compactly supported probability measures $\rho_1, \rho_2, \ldots, \rho_k.$  The vector of linear statistics $(\PCLS{p_1}{\rho_1},\ldots, \PCLS{p_k}{\rho_k})$ converges in distribution as $L\to\infty$ to a mean $0$ Gaussian vector $(\xi_1,\xi_2,\ldots,\xi_k)$ with covariance
\[
\Ex \xi_{i} \xi_j
=\frac{8\mu \nu}{\beta \pi} 
\int_{\H^2}
p_i'(x(r_ie^{i\theta_i}))
p_j'(x(r_je^{i\theta_i}))
K(r_ie^{i\theta_i},r_je^{i\theta_j})~
dM_{i,j},
\]
where $dM_{i,j}$ is the measure given in polar coordinates by
\[
dM_{i,j}=
r_i \sin \theta_i
r_j \sin \theta_j
d\theta_i d\theta_j d \rho_i( r_i ) d \rho_j(r_j).
\]
\end{proposition} 

We then show that the height function has a type of a priori stability in the sense that the standard deviation of a linear statistic can be controlled by a suitable seminorm of the test function.  Fractional Sobolev norms are particularly useful in this regard.  Previous work of Shcherbina shows that this norm can be used to control the variance of a linear statistic of a single sample covariance matrix~\cite{Shcherbina11}.% \textcolor{red}{IS THIS REFERENCE RIGHT?}
\begin{proposition}
\label{prop:shcher_var}
Suppose Assumption~\ref{a:4e} on $\left\{ \atomvar_{i,j} \right\}$ holds.
%for some $\epsilon > 0,$ and let $w = 
%\sup_{i,j} \Ex \left| \atomvar_{i,j} \right|^{4+\epsilon}.$
%Suppose in addition that
%\[
%	\sup_{i,j} \frac{\Ex \left| \atomvar_{i,j} \right|^{6}}{L^{1-\tfrac \epsilon 2}} < w
%	\quad \text{ and } \quad
%	\sup_{i,j} \frac{\Ex \left| \atomvar_{i,j} \right|^{8}}{L^{2-\tfrac \epsilon 2}} < w,
%\]
%where we now allow the $\left\{ \atomvar_{i,j} \right\}$ to carry some $L$-dependence.
Let $\beta=1$ and let $Q$ be an $m(L) \times L$ submatrix of $\ambientmatrix$ such that $m(L)/L \to c \geq 1.$
For every $s > \tfrac{3}{2},$ there is a constant $C_s=C_s(w,\epsilon) > 0$ so that 
\[
	\Var X_\phi^{Q^*Q/L} \leq C_s \FSN[\phi][s]^2
\]
for all $L \geq 1$ and all $\phi.$
\end{proposition}
\noindent See Proposition 4 and Lemma 2 of~\cite{Shcherbina11}.  %The extra moment conditions on $\left\{ \atomvar_{i,j} \right\}$ arise naturally when truncating variables satisfying Assumption~\ref{a:4e} to be supported on an interval of length $O(L^{1/2-\epsilon/2}).$
% \textcolor{red}{WHAT ABOUT THIS?}

As a consequence, we also get a similar type of bound that holds for any planar test statistic. As a simple application of Jensen's inequality, we will show that for general planar statistics:%\textcolor{red}{WHAT'S THE DIFF BETWEEN PROP AND LEMMA?}
\begin{lemma}
\label{lem:rho_var}
%Under the same assumptions on $\left\{ \atomvar_{i,j} \right\}$ in Proposition~\ref{prop:shcher_var},
Suppose Assumption~\ref{a:4e} on $\left\{ \atomvar_{i,j} \right\}$ holds.
For every $s > \tfrac{3}{2},$ there is a constant $C_s > 0$ sufficiently large that 
\[
\Var\PCLSg{\psi\rho}
\leq 
C_s \FSNr[\psi][s][\rho]^2.
\]
\end{lemma} 

As it stands, the functions in Proposition~\ref{prop:CLSplanar} do not actually have finite $\FSNr$ norm, as polynomials do not have finite $\FSN$ norm.  Thus, the last step involves truncating the polynomials with suitable cutoff functions, after which point, we will obtain a dense class of functions, and we may then prove Theorem~\ref{thm:pcls}.

The rest of the paper is structured as follows: Sections \ref{sec:covar} and \ref{sec:pclt} deal with the proofs of Propositions \ref{prop:CLS} and \ref{prop:CLSplanar}; specifically, Section \ref{sec:covar} contains the calculation of the covariance only, while Section \ref{sec:pclt} contains the calculation of the other moments as well as the proofs of the two Propositions. Finally, Section \ref{sec:extension} presents the extension argument and finishes the proof of Theorem \ref{thm:pcls}.

\subsection*{Truncation}
	\newcommand{\trambient}{\hat G}
	\newcommand{\trvar}{\hat Z}
	\newcommand{\trsmf}{\hat{\mathbf{B}}({y})}
	We will use a truncation argument to pass from the $4$-th moment hypothesis in Assumption~\ref{a:4e} to a family of variables with bounded, $L$-dependent support.
	Fix $\delta >0$ and define variables $\trvar_{i,j}=\trvar_{i,j}(L)$ so that 
	\begin{enumerate}
	\item
		$\trvar_{i,j} = \atomvar_{i,j}$ provided $|\atomvar_{i,j}| \leq L^{1/2-2\delta},$
	\item
		$|\trvar_{i,j}| \leq L^{1/2-\delta}$ almost surely,
	\item and
		$\Exp \trvar_{i,j} = 0,$
		$\Exp |\trvar_{i,j}|^2 =1,$ and
		$\Exp |\trvar_{i,j}|^4 =1+2/\beta.$
	\end{enumerate}
	Such a truncation is always possible for $L$ sufficiently large.
	%Let $\delta >0$ be a constant to be determined later and set $$\trvar_{i,j}=\atomvar_{i,j}\one[|\atomvar_{i,j}| < L^{1/2-\delta}].$$
	Let $\trambient$ be the $[\mu L] \times [\nu L]$ upper left corner of $( \trvar_{i,j} )_{i,j \geq 1}$.
	%As $\rho$ is compactly supported, we may find a $K>0$ so that $[0,K]$ contains its support.  
	From the moment hypotheses on $\atomvar,$
	\[
		\Pr\left[
			\trvar_{i,j} \neq \atomvar_{i,j}
		\right]
		\leq
		\Pr\left[
			|\atomvar_{i,j}| \geq L^{1/2-2\delta}
		\right]
		\leq \frac{\Exp |\atomvar_{i,j}|^{4+\epsilon}}{L^{(1/2-2\delta)(4+\epsilon)}}.
	\]
	By making $\delta > 0$ sufficiently small, we may make this bound $o(L^2).$ We also get the following higher moment bounds for $\trvar_{i,j}$ for all $k \geq 5$
	\begin{equation}
		\label{eq:higher}
		\sup_{i,j,L}\Exp |\trvar_{i,j}|^k L^{-(\tfrac12-\delta)(4-k)} < \infty.
	\end{equation}
%	\textcolor{red}{INCONSISTENCY: $\trambient$ is infinite, but $G$ is finite.}

	Define $\trsmf$ to be the submatrix of $\trambient$ where $i,j$ run over $[y \mu L] \times [y \nu L].$  This is the exact analogue of $\smf{y},$ and we have that
	\[
		\Pr\left[
			\exists y \in [0,K] : \trsmf \neq \smf{y}
		\right] \to 0
	\]
	as $L\to \infty,$ since $\{\trsmf\}_{y \leq K}$ depends on at most $O(L^2)$ entries of $\trambient.$  Thus, it suffices to prove the theorem for the matrix of truncated variables $\trambient$, as then the conclusion holds for $\ambientmatrix$. From hereon, we will assume that we have already replaced $\ambientmatrix$ by $\trambient$ and $\left\{ \atomvar_{i,j} \right\}$ by $\left\{ \trvar_{i,j} \right\}$ to lighten the notation.

\section{Covariance Calculation}
\label{sec:covar}
The goal of this section is to calculate the limiting covariance $\cmti{k}{1}$ and $\cmti{k}{2}.$  We will do the calculation fully for the $\beta=1$ case and point out the differences for $\beta=2$ and $\beta=4$ along the way.

In this section we \textbf{do not assume} the \emph{exact} $4$-th moment condition on $\atomvar_{i,j}$ given in Assumption \ref{a:4e}.  Instead, we will assume $\Exp Z_{i,j}^4 = \Exp Z_{1,1}^4$ for all $i,j$ and $\Exp Z_{1,1}^4<\infty.$ We will see how the exact form of the $4$-th moment necessary for obtaining the Gaussian free field in the limit emerges from this context.

Throughout, we will use the following notation.  Define
\begin{align*}
%\label{narayanas}
\nare{k}{\gamma} &\Def \sum_{T} \gamma^{e(T)} \\
\naro{k}{\gamma} &\Def \sum_{T} \gamma^{o(T)},
\end{align*}
where the sums are over all rooted plane trees $T$ with $k$ edges.  The statistics $e(T)$ and $o(T)$ are the number of vertices in these plane trees at even and odd depth from the root.  The polynomials $\naro{k}{\gamma}$ are precisely the Narayana polynomials, for all $k \geq 0$.  The $\nare{k}{\gamma}$ polynomials are also the Narayana polynomials for all $k > 0,$ while $\nare{0}{\gamma} = \gamma$ (see \cite[Remark p.176]{Deutsch} or \cite{Kreweras}).%  \textcolor{red}{REFERENCE!}

From these definitions, we may take advantage of the well-known generating function for Narayana polynomials
\[
F(z, \gamma) \Def \sum_{z=0}^\infty z^k \naro{k}{\gamma} =  
{\frac {z \left( 1-\gamma \right) +1-\sqrt { \left( z \left( 1-
\gamma \right) +1 \right) ^{2}-4\,z}}{2z}}.
\]
Recall that a centered trace of a power of $\submatrix_1^* \submatrix_1$, 
\[
\cmti{k}{1} = 
\tr \left(\submatrix_1^* \submatrix_1\right)^{k}
-\Exp \tr \left(\submatrix_1^* \submatrix_1\right)^{k}~,
\]
can be written as a sum over closed walks of length $2k$ on the complete bipartite graph $K(\mathbb{N},\mathbb{N}).$  Let $\walkset{k}$ be those closed walks so that
\[
\cmti{k}{1} = \sum_{w \in \walkset{k}} \atomword{w},
\]
where we formally represent a walk as $w~:~[2k] \to \mathbb{N}$ and define
\[
\atomword{w} \Def \prod_{i=1}^{k} 
\atomvar_{w(2i-1), w(2i)} 
\overline{\atomvar_{w(2i+1), w(2i)}}
-\Exp \prod_{i=1}^{k} 
\atomvar_{w(2i-1), w(2i)} 
\overline{\atomvar_{w(2i+1), w(2i)}}.
\]

The covariance of $\cmti{k}{1}$ and $\cmti{l}{2}$ is therefore given by
\[
  \Exp \cmti{k}{1} \cmti{l}{2} =
  \sum_{\substack{w_1 \in \walkset{k} \\w_2 \in \walkset{l}}}
  \Exp \atomword{w_1} \atomword{w_2}.
\]
%From the counting, each new vertex in the walk contributes a factor of $L$ (from the $\Theta(L)$ choices for its placement), and each step in the walk contributes a factor of order $L^{-1/2}.$  
If any edge is visited only once by $w_1$ and $w_2$, then its contribution to the covariance is $0,$ as taking expectations causes the term to vanish.  Thus, all contributing pairs of walks cover every edge at least twice.  Also, if the walks $w_1$ and $w_2$ have disjoint edge supports, then by the independence of $\atomword{w_1}$ and $\atomword{w_2},$ the contribution of such pairs $(w_1,w_2)$ to the covariance is $0.$

As each walk traces out a connected graph, and the graph sum of the two walks has at most $k+l+2$ vertices.  In the extreme case, the graph is a forest with two trees, the paths have disjoint edge supports, and the contribution of these paths to the covariance is $0.$  If the support of the pair has $k+l+1$ vertices, the walks share a vertex and not an edge, and the contribution is again $0$.  Further, walks that cover strictly fewer than $k+l$ vertices provide a negligible contribution (see the proof of Proposition~\ref{prop:CLS} for details).
%, as for any pair of walks, $|\Exp \atomword{w_1} \atomword{w_2}| = O(L^{-k-l}),$ but there are only $O(L^{k+l-1})$ many walks on $k+l-1$ or fewer vertices.  
Thus, the only terms that contribute to the limit are those pairs of walks that cover exactly $k+l$ vertices and necessarily traverse a common undirected edge of $K(\mathbb{N},\mathbb{N}).$

The pairs of walks of this form arise in one of two ways.  The first possibility is that both of the pairs are depth first search walks of some trees that cover a common edge.  The second is that both walks cover a unicyclic graph, traversing the cycle once and making excursions along trees that are attached to this cycle.  The cycle must be common to both walks to ensure that it is traversed twice by the union of walks. This much is identical to what occurs for Wigner matrices, see \cite[Lemma 2.1.33]{AnGuZe} for proofs.

The pairs of walks of the first type provide a contribution that is, to first order in $L,$
\begin{equation*}
T_1 \Def L^{-k-l}\left[\Ex |\WEV|^4 - 1\right]\sum_{T_k,T_l} k\cdot l\cdot m_{1,2}n_{1,2}m_1^{e(T_k)-1} n_1^{o(T_k)-1} m_2^{e(T_l)-1} n_2^{o(T_l)-1}.
\end{equation*}
This expression follows from a simple counting argument.  There are $k\cdot l$ ways to glue the rooted trees $T_k$ and $T_l$ together along an edge, since having picked an edge from either tree, there is exactly one choice of orientation to pick so that roots are an even number of steps apart.  Having glued the two trees together, there are asymptotically
\[
m_{1,2}n_{1,2}m_1^{e(T_k)-1} n_1^{o(T_k)-1} m_2^{e(T_l)-1} n_2^{o(T_l)-1}
\]
ways to label the vertices of the paths.
From the limiting relationships for the ``$m$'' and ``$n$'' sequences, we have that
\begin{equation}
\label{eq:wish_cov_glued_tree}
T_1 \tendsto \nu_1^k \nu_2^l \left(\Ex |\WEV|^4 - 1\right) (\theta \nu_1\nu_2) \cdot k \cdot l \sum_{T_k,T_l} \gamma_1^{e(T_k)} \gamma_2^{e(T_l)}.
\end{equation}

The contribution of the pairs of unicyclic walks is, to first order in $L,$
\begin{equation}
\label{eq:wish_cov_glued_cycles}
T_2 \Def L^{-k-l}\sum_{\substack{r=4, \\ r\text{ even}}}^\infty \frac{1}{r/2} 2\left(\frac{m_{1,2}n_{1,2}}{m_1m_2}\right)^{r/2} n_1^k n_2^l \sum_{ \substack{\mu \models k-r/2 \\ \lambda \models l-r/2 } } ~~~\sum_{\substack{T_{\mu_1},\ldots,T_{\mu_r} \\ T_{\lambda_1},\ldots,T_{\lambda_r}}} k \cdot l \cdot \left(\frac{m_1}{n_1}\right)^{\ell_1}\left(\frac{m_2}{n_2}\right)^{\ell_2},
\end{equation}
where $\ell_1 = e(T_{\mu_1}) + o(T_{\mu_2}) + e(T_{\mu_3}) + \cdots + o(T_{\mu_r}),$ where $\ell_2 = e(T_{\lambda_1}) + o(T_{\lambda_2}) + e(T_{\lambda_3}) + \cdots + o(T_{\lambda_r}),$ and where $\mu$ and $\lambda$ are $r$-compositions of $k-r/2$ and $l-r/2$ respectively.

This contribution of the paths glued along their common $r$-cycle can be counted by the following procedure.
\begin{enumerate}
\item Mark one of the ``m''-side vertices of $w_1$ cycle to break symmetry.  There are $r/2$ choices for this mark, and after removing this mark, we will have overcounted exactly $r/2$-fold.
\item At each of the $r$-vertices of the cycle, choose the number of edges for the pendant tree that will dangle from this vertex.  These choices are $\mu$ and $\lambda$ respectively.
\item Orient the $w_2$ path cycle to either match the orientation of $w_1$ or to oppose it, for an extra factor of $2.$
\item Choose the starting \emph{step} for each of the walks.  The walks have $2k$, respectively, $2l$ steps, but only steps originating at an ``m''-side vertex can be starting locations.  Thus, there are $k\cdot l$ many such choices.
\item Label the vertices.  For a given choice of trees, there are asymptotically
\[
m_{1,2}^{r/2}n_{1,2}^{r/2}m_1^{\ell_1 - r/2}n_1^{k- \ell_1} m_2^{\ell_2 - r/2} n_2^{k-\ell_2}
\]
many ways to label the paths.
\end{enumerate}

\begin{remark}
  \label{rem:other_beta}
  The only place where the cases $\beta=2,4$ differ from $\beta=1$ is in the contribution of paths whose cycles are aligned in the same direction.   In the complex case, this contributes a factor of $Z_{i,j}^2$ to $\atomword{w_1}\atomword{w_2},$ which vanishes in expectation.  Thus, for $\beta=2,$ only pairs of paths whose cycles are counter-aligned (aligned in opposite direction) contribute to the covariance, so that the contribution of step (3) should be replaced by $1$.  In the quaternionic case, one has $\Exp Z_{i,j}^2 = -\frac{1}{2},$ so that the contribution of step (3) should be replaced by $1 + (-\tfrac{1}{2})^r.$  Thus, in the complex case, $T_2$ is halved, and the calculations that follow for the real case are directly applicable.  In the quaternion case, the calculations that follow need some minor modifications to show the covariance carries the desired $\frac{1}{\beta} = \frac{1}{4}$ factor.
\end{remark}

These limiting formulae will now be recast in terms of generating functions and Narayana polynomials.  
From the generating function for the Narayana polynomials, this may be written as
\[
T_1 \tendsto \nu_1^k \nu_2^l \left(\Ex \WEV^4 - 1\right) (\theta \nu_1\nu_2)  \cdot k \cdot l \left[ z_1^k z_2^l\right] F(z_1,\gamma_1)F(z_2,\gamma_2),
\]
where $[z_1^k z_2^l]$ denotes the $z_1^kz_2^l$ coefficient of the following series.

For $T_2$, we begin by rewriting the limiting expression in terms of the Narayana polynomials,
\begin{multline}
\sum_{\substack{T_{\mu_1},\ldots,T_{\mu_r} \\ T_{\lambda_1},\ldots,T_{\lambda_r}}} \left(\gamma_1\right)^{\ell_1}\left(\gamma_2\right)^{\ell_2}
= \\
%\gamma_1\gamma_2\partial_{\gamma_1}
\left(
\nare{\mu_1}{\gamma_1}
\naro{\mu_2}{\gamma_1}
\nare{\mu_3}{\gamma_1}
\cdots
\naro{\mu_r}{\gamma_1}\right)
%\partial_{\gamma_2}
\left(
\nare{\lambda_1}{\gamma_2}
\naro{\lambda_2}{\gamma_2}
\nare{\lambda_3}{\gamma_2}
\cdots
\naro{\lambda_r}{\gamma_2}
\right).
\end{multline}
In \eqref{eq:wish_cov_glued_cycles}, this is expression is summed over all compositions which may be recast efficiently in terms of coefficient extraction from products of generating functions.  Explicitly,
%\begin{multline}
%\sum_{ \substack{\mu \models k-r/2 \\ \lambda \models l-r/2 } }
%%\gamma_1\gamma_2\partial_{\gamma_1}
%\left(
%\nare{\mu_1}{\gamma_1}
%\naro{\mu_2}{\gamma_1}
%\nare{\mu_3}{\gamma_1}
%\cdots
%\naro{\mu_r}{\gamma_1}\right)
%%\partial_{\gamma_2}
%\left(
%\nare{\lambda_1}{\gamma_2}
%\naro{\lambda_2}{\gamma_2}
%\nare{\lambda_3}{\gamma_2}
%\cdots
%\naro{\lambda_r}{\gamma_2}
%\right)
%= \\
\begin{multline}
\sum_{ \substack{\mu \models k-r/2 \\ \lambda \models l-r/2 } }
\prod_{i=1}^{r/2} 
\nare{\mu_{2i-1}}{\gamma_{2i-1}}
\nare{\mu_{2i}}{\gamma_{2i-1}}
\nare{\lambda_{2i-1}}{\gamma_{2i}}
\nare{\lambda_{2i}}{\gamma_{2i}}
 \\
=\left[z_1^{k-r/2}z_2^{l-r/2}\right]
\left(
 G(z_1,\gamma_1)^{r/2} G(z_2,\gamma_2)^{r/2}
\right),
\end{multline}
where $G(z,\gamma) = F(z,\gamma)( \gamma - 1 + F(z,\gamma))$ is the product of the generating functions for $\naro{k}{\gamma}$ and $\nare{k}{\gamma}$, respectively.  Summing over all $r$ and reindexing the sum, this shows that \eqref{eq:wish_cov_glued_cycles} can be written as
\begin{align*}
T_2 &\tendsto 
2 k\cdot l \cdot \nu_1^k\nu_2^l\sum_{r=2}^\infty \frac{\left( \theta \nu_1\nu_2 \right)^r}{r}
\left[z_1^{k-r}z_2^{l-r}\right]
%\gamma_1\gamma_2\partial_{\gamma_1}\partial_{\gamma_2}
\left(
 G(z_1,\gamma_1)^{r} G(z_2,\gamma_2)^{r}
\right) \\ 
&=
\left[z_1^{k}z_2^{l}\right]
2 k \cdot l \cdot \nu_1^k\nu_2^l\sum_{r=2}^\infty \frac{\left( \theta \nu_1\nu_2 \right)^r}{r}
%\gamma_1\gamma_2\partial_{\gamma_1}\partial_{\gamma_2}
\left(
 z_1^{r}G(z_1,\gamma_1)^{r} z_2^rG(z_2,\gamma_2)^{r}
\right) \\ 
&=
-\left[z_1^{k}z_2^{l}\right]
2k \cdot l \cdot 
\nu_1^k\nu_2^l
\cdot
\phi\left( \theta \nu_1\nu_2 \cdot z_1G(z_1,\gamma_1) z_2G(z_2,\gamma_2)
\right),
\end{align*}
where $\phi(z) = \log(1-z) + z$, and provided that $\left| \theta \nu_1 \nu_2 z_1 G(z_1, \gamma_1) z_2 G(z_2, \gamma_2) \right| < 1.$ 

Define $y_1 = z_1 G(z_1,\gamma_1)$ and define $y_2$ analogously.  From elementary operations, it can be checked that
\[
z_1^{-1} = \frac{\gamma_1 + (1+\gamma_1)y_1 + y_1^2}{y_1}.
\]
This allows the coefficient extraction to be represented in terms of contour integrals as
\[
T_2 \tendsto -\frac{2 \nu_1^k \nu_2^l k \cdot l}{(2\pi i)^2} \DBLoint
\frac{\phi\left( \theta \nu_1\nu_2 y_1 y_2 \right)}{z_1^{k+1}z_2^{l+1}}
dz_1\,dz_2,
\]
where the contours for $z_1$ and $z_2$ wind once around the origin in the positive orientation and have $|\theta \nu_1 \nu_2 y_1 y_2| < 1.$  This same notation can be used to write the limiting covariance for $T_1.$  It can be verified that $F(z,\gamma) = 1 + zG(z,\gamma),$ so that the limiting expression, in terms of a contour integral, is
\[
T_1 \tendsto \frac{\nu_1^k \nu_2^l\left(\Ex \WEV^4 - 1\right) (\theta \nu_1\nu_2)  k\cdot l}{(2\pi i)^2} \DBLoint
\frac{1}{z_1^{k+1}z_2^{l+1}}
(1+y_1)(1+y_2)dz_1\,dz_2.
\]
As both $k,l > 0,$ iterating the integration and applying the residue theorem shows this is equal to
\[
T_1 \tendsto \frac{\nu_1^k \nu_2^l\left(\Ex \WEV^4 - 1\right) (\theta \nu_1\nu_2)  k\cdot l}{(2\pi i)^2} \DBLoint
\frac{1}{z_1^{k+1}z_2^{l+1}}
y_1y_2\,dz_1\,dz_2.
\]
This is also easily argued combinatorially in terms of coefficient extraction and the relation between $F(z,\gamma)$ and $G(z,\gamma).$  

Changing the integration to be over $y_1$ and $y_2,$ the two limiting expressions become
\begin{align}
\label{wish_prim_cov}
T_1 &\tendsto \left(\Ex \WEV^4 - 1\right)\frac{\nu_1^k \nu_2^l k \cdot l}{(2\pi i)^2} 
\DBLoint
\frac{ \theta \nu_1\nu_2 y_1 y_2}{z_1^{k-1}z_2^{l-1}}
\left( 
1-\frac{\gamma_1}{y_1^2}
\right)
\left( 
1-\frac{\gamma_2}{y_2^2}
\right)
dy_1\,dy_2. \\
T_2 &\tendsto -\frac{2\nu_1^k \nu_2^l k \cdot l}{(2\pi i)^2} 
\DBLoint
\frac{\phi\left( \theta \nu_1\nu_2  y_1 y_2 \right) }{z_1^{k-1}z_2^{l-1}}
\left( 
1-\frac{\gamma_1}{y_1^2}
\right)
\left( 
1-\frac{\gamma_2}{y_2^2}
\right)
dy_1\,dy_2.
\end{align}
After summing these two terms, we split the resulting covariance into two expressions $\WCM_{k,l}$ and $\WEM_{k,l}$, the first representing the contribution of the free field and the other contributing an error term that vanishes when the matrix entries agree with the Gaussian's $4^{th}$ moment.
\begin{align}
\label{wish_cov_defs}
\WCM_{k,l} &\Def -\frac{2 \nu_1^k \nu_2^l k \cdot l}{(2\pi i)^2} 
\DBLoint
\frac{\log\left(1- \theta \nu_1\nu_2 y_1 y_2\right)}{z_1^{k-1}z_2^{l-1}}
\left( 
1-\frac{\gamma_1}{y_1^2}
\right)
\left( 
1-\frac{\gamma_2}{y_2^2}
\right)
dy_1\,dy_2, \\
\WEM_{k,l} &\Def
\left(\Ex \WEV^4 - 3\right)
\frac{\nu_1^k \nu_2^l k \cdot l}{(2\pi i)^2} 
\DBLoint
\frac{{\theta}{\nu_1\nu_2} y_1 y_2}{z_1^{k-1}z_2^{l-1}}
\left( 
1-\frac{\gamma_1}{y_1^2}
\right)
\left( 
1-\frac{\gamma_2}{y_2^2}
\right)
dy_1\,dy_2.
\end{align}
We emphasize that the contours are any that wind once positively around $y_1=0$ and $y_2 =0,$ and keep 
\begin{eqnarray} \label{cond_gamma}
\left| \theta \nu_1 \nu_2 y_1 y_2 \right| \leq 1~.
\end{eqnarray}
Note in the proof below that equality in the above is allowable.

\begin{lemma}
\label{lem:wish_cov_GF}
The two expressions above are equivalent to
\begin{equation*}
  %\label{final_expr}
\WCM_{k,l}
 =  -\frac{4 k\cdot l}{\pi^2} 
\!\!
\DBLoint
\!\!
\left(\mu_1 \!+\! \nu_1 \!+\! 2\Re \zeta_1\right)^{k-1}
\!
\left(\mu_2 \!+\! \nu_2 \!+\! 2\Re \zeta_2\right)^{l-1}
\!
K(\zeta_1,\zeta_2)
\frac{\Im \zeta_1}{\zeta_1}
\frac{\Im \zeta_2}{\zeta_2}
d\zeta_1\,d\zeta_2,
\end{equation*}
respectively, 
\begin{equation*}
	\WEM_{k,l} = - \frac{
	4\theta \left(\Ex \WEV^4 - 3\right)
	k \cdot l
	}{
		\pi^2 
	}
\DBLoint
\!\!
\left(\mu_1 \!+\! \nu_1 \!+\! 2\Re \zeta_1\right)^{k-1}
\!
\left(\mu_2 \!+\! \nu_2 \!+\! 2\Re \zeta_2\right)^{l-1}
\!
\frac{\left(\Im \zeta_1\right)^2}{\zeta_1}
\frac{\left(\Im \zeta_2\right)^2}{\zeta_2}
d\zeta_1\,d\zeta_2.
\end{equation*}
Here
\(
K(\zeta_1,\zeta_2)
=
\log \left|
\frac{
{\theta^{-1} - \zeta_1 \zeta_2}
}
{
{\theta^{-1}- \zeta_1 \bar{\zeta_2}}
}
\right|
\),
%$\theta' = \tfrac{\mu_1\nu_1\mu_2\nu_2}{\theta}$ 
and the contours are semicircles centered at $0$ in the upper half plane of radii $\sqrt{\mu_i \nu_i}$ for $i=1,2.$
\end{lemma} 
\begin{proof}
We start by noting that, that if $\xi = R e^{i \theta}$ for real $R$ and $\theta$ then, $1 - \frac{R^2}{\xi^2} = 2i \frac{\Im(\xi)}{\xi}$. We will now choose the integration curves in \eqref{wish_cov_defs} to be the circles $|y_1| = \sqrt{\gamma_1}$, respectively, $|y_2| = \sqrt{\gamma_2}$. Note that for this choice of curves to be allowable it suffices that
\[
\theta \nu_1 \nu_2 \sqrt{\gamma_1} \sqrt{\gamma_2} = \theta \sqrt{\mu_1 \nu_1 \nu_2 \mu_2}<1\,.
\]
In the case that $\theta \sqrt{\mu_1 \nu_1 \nu_2 \mu_2} =1,$ we may take circles of slightly smaller radius and take the limits as the contours' radii go to $\sqrt{\gamma_i}$ respectively.  By the integrability of the logarithmic singularity, a dominated convergence argument shows that the formula holds in this case as well. So as long as condition \eqref{cond_gamma} is satisfied, the results of this lemma will be valid.  
%This is equivalent to having $\theta^2 \mu_1 \nu_1 \mu_2 \nu_2 <1$; if we look at the definition of $\theta$, we can see that this will be fulfilled for every choice of limiting sizes \emph{except} when $\theta = \mu_1 = \nu_1 = \mu_2 = \nu_2 = 1$, which we have explicitly forbidden\footnote{Must forbid this.}. Thus the choice of curves is acceptable.

With this choice of curves, \eqref{wish_cov_defs} becomes 
\[
\WCM_{k,l} = \frac{8\nu_1^k \nu_2^lk \cdot l}{(2\pi i)^2} 
\DBLoint_{\substack{ |y_i| = \sqrt{\gamma_i} \\ i=1,2}}
\frac{\log\left(1- \theta \nu_1\nu_2 y_1 y_2\right)}{z_1^{k-1}z_2^{l-1}}
\frac{\Im(y_1)}{y_1}\frac{\Im(y_2)}{y_2} dy_1\,dy_2~;
\]
note now that $\frac{\nu_i}{z_i} = \mu_i + \nu_i + 2 \nu_i \Re y_i$ for $i = 1,2$. 
Similarly, one obtains that
\[
\WEM_{k,l} =  -\left(\Ex \WEV^4 - 3\right)
\frac{4 \nu_1^k \nu_2^l k \cdot l}{(2\pi i)^2} 
\DBLoint_{\substack{ |y_i| = \sqrt{\gamma_i} \\ i=1,2}}
\frac{{\theta}{\nu_1\nu_2} y_1 y_2}{z_1^{k-1}z_2^{l-1}}
\frac{\Im(y_1)}{y_1}\frac{\Im(y_2)}{y_2} dy_1\,dy_2~.
\]

Focusing on $\WCM$, we split the integral into four parts, corresponding to the choice of upper or lower semicircle for each one of $y_1$ and $y_2$.  We can transform each of these integrals to an integral over the cartesian product of upper semicircles. This transformation does not in any way affect $z_1$ and $z_2$, as only the real parts are involved in the $z_i$s.
Crucially, on the circle $|y|=R$, the transformation $y \mapsto \bar{y}$, which takes upper semicircle into the lower one and vice-versa, has the property that 
\[
  \frac{\Im y}{y} dy  = \frac{\Im ({\overline{y})}}{\bar{y}} d\bar{y},
\]
as $y \bar{y} = R^2$ implies $dy/y = - d\bar{y}/\bar{y}$, and $\Im y = -\Im \bar{y}$. 

As in \cite{Borodin}, note that 
%{\footnotesize \[
%\log \left (1 - \theta \nu_1 \nu_2 y_1 y_2 \right ) - \log \left ( 1 -  \theta \nu_1 \nu_2 y_1 \bar{y_2} \right) - \log \left ( 1 - \theta \nu_1 \nu_2 \bar{y_1} y_2 \right) + \log \left ( 1 - \theta \nu_1 \nu_2 \bar{y_1} \bar{y_2}\right)  = 2 \log \left | \frac{1 - \theta \nu_1 \nu_2 y_1 y_2}{1 - \theta \nu_1 \nu_2 y_1 \bar{y_2}} \right |~.
%\]}
\begin{align*}
2 \log \left | \frac{1 - \theta \nu_1 \nu_2 y_1 y_2}{1 - \theta \nu_1 \nu_2 y_1 \bar{y_2}} \right | = 
&\log \left (1 - \theta \nu_1 \nu_2 y_1 y_2 \right ) 
- \log \left ( 1 -  \theta \nu_1 \nu_2 y_1 \bar{y_2} \right) \\
- &\log \left ( 1 - \theta \nu_1 \nu_2 \bar{y_1} y_2 \right) 
+ \log \left ( 1 - \theta \nu_1 \nu_2 \bar{y_1} \bar{y_2}\right). 
\end{align*}
\normalsize Finally, because the transformation $y \mapsto \bar{y}$ reverses direction and introduces a minus sign in integration, we can conclude that the double integral over the full circles can be rewritten, in the case of $\WCM_{k,l}$, as 
\[
\WCM_{k,l} = -\frac{4\nu_1^k \nu_2^l k \cdot l}{\pi^2} 
\DBLoint_{\substack{
|y_i| = \sqrt{\gamma_i} \\
 \Im y_i >0  \\
 i=1,2
}}
\frac{1}{z_1^{k-1}z_2^{l-1}}    \log \left | \frac{1 - \theta \nu_1 \nu_2\cdot y_1 y_2}{1 - \theta \nu_1 \nu_2 \cdot y_1 \bar{y_2}} \right |
\frac{\Im(y_1)}{y_1}\frac{\Im(y_2)}{y_2} dy_1\,dy_2~; 
\]

In the case of $\WEM_{k,l}$, we note that 
\[
  -4\Im(y_1)\Im(y_2) = (y_1 - \bar{y_1})(y_2 - \bar{y_2}) = y_1 y_2 - \bar{y_1} y_2 -y_1 \bar{y_2} + \bar{y_1} \bar{y_2}~.
\]
Thus following the same argument
\[
\WEM_{k,l} = - \left(\Ex \WEV^4 - 3\right)
\frac{4\nu_1^k \nu_2^l k \cdot l}{\pi^2} 
\DBLoint_{\substack{
|y_i| = \sqrt{\gamma_i} \\
 \Im y_i >0  \\
 i=1,2
}}\frac{{\theta}{\nu_1\nu_2}}{z_1^{k-1}z_2^{l-1}} \frac{\left ( \Im(y_1)\right)^2}{y_1}\frac{\left(\Im(y_2)\right)^2}{y_2} dy_1\,dy_2~.
\]
When the $4$-th moment condition in Assumption~\ref{a:4e} is fulfilled, $\WEM_{k,l}$ disappears. Its presence when this condition is not fulfilled introduces a correction to the Gaussian free field.

After one final change of variables $y_i \mapsto \nu_i^{-1} \zeta_i$, dividing the top and bottom of the fraction under the logarithm by $\theta$ (in the expression for $\WCM_{k,l}$), the lemma (under the $4$-th moment condition in Assumption \ref{a:4e}) is proved.
\end{proof}

\section{Polynomial CLTs} \label{sec:pclt}
In this section, we will show how the central limit theorems for polynomial test functions (Propositions~\ref{prop:CLS} and \ref{prop:CLSplanar}) are derived.  The principal difficulty for both is the covariance calculation, made in the previous section.  The remainder of the work is to show that the limiting mixed moments agree with that of a Gaussian.  In the case that $\left\{ \atomvar_{i,j} \right\}$ have uniformly bounded moments of all orders, this proof is essentially a minor modification to the corresponding proof for a single Wigner matrix (see (2.1.46) of~\cite{AnGuZe} and the development thereafter).

\begin{proof}[Proof of Proposition~\ref{prop:CLS}]
As in Section~\ref{sec:covar}, we recall that a centered trace of a power, 
\[
\cmti{p_j}{j} = 
\tr \left(\submatrix_j^* \submatrix_j\right)^{p_j}
-\Exp \tr \left(\submatrix_j^* \submatrix_j\right)^{p_j}~,
\]
can be written as a sum over closed walks of length $2p_j$ on the complete bipartite graph $K(\mathbb{N},\mathbb{N}).$  Let $\walkset{j}$ be those closed walks so that
\[
\cmti{p_j}{j} = \sum_{w \in \walkset{j}} \atomword{w},
\]
where we formally represent a walk as $w~:~[2p_j] \to \mathbb{N}$ and recall that
\[
\atomword{w} = \prod_{i=1}^{p_j} 
\atomvar_{w(2i-1), w(2i)} 
\overline{\atomvar_{w(2i+1), w(2i)}}
-\Exp \prod_{i=1}^{p_j} 
\atomvar_{w(2i-1), w(2i)} 
\overline{\atomvar_{w(2i+1), w(2i)}}.
\]
Explicitly, $\walkset{j}$ is the set of all closed walks with $2j$ steps on the complete bipartite graph $K(M_j,N_j)$ starting on the $M_j$ side, where $M_j$ and $N_j$ are the respective row indices and column indices of $\submatrix_j$ in the infinite array $(\atomvar_{i,j}).$  Without loss of generality, we assume that all such $M_j$ and $N_j$ are contained in $[L].$  

The method of proof will be the computation of moments.  If suffices to show that for each word $m \in \left[ k \right]^l,$ the mixed moments satisfy
\[
\expect \prod_{j=1}^l \cmti{p_{m(j)} }{m(j)} = \begin{cases}
o(1) & \text{if $l$ is odd,} \\
o(1)+ \sum_{G} \prod_{\{a,b\} \in \mathcal{E}(G)} \expect \cmti{p_{m(a)}}{m(a)} \cmti{p_{m(b)}}{m(b)} & \text{if $l$ is even,}
\end{cases}
\]
where the sum is over all graphs $G$ that are perfect matchings on the vertices $[k]$ and where $\mathcal{E}(G)$ is the edge set of this graph.  By Lemma~\ref{lem:wish_cov_GF} it will then follow that this moment converges to the desired Gaussian moments.

Consider a fixed collection of walks $(w_1, w_2, \ldots, w_l) \in \prod_{i=1}^l \walkset{m(i)}.$  Define the \emph{support} $G$ of this collection of walks to be the subgraph of $K(L,L)$ that contains precisely those edges traversed by the walk.  Call two tuples of walks \emph{isomorphic}, if it is possible to permute $[L]$ to realize one in terms of the other.

Let
\(
p \Def \sum_{i=1}^{l} p_{m(i)},
\)
so that there are exactly $2p$ directed edges (with multiplicity) used by $(w_1,w_2,\ldots, w_l).$
If somewhere in this collection of walks there is an edge which is traversed only once (in total, by all the walks), then it follows that 
\[
\Exp \prod_{i=1}^l \atomword{w_i} = 0,
\]
and hence it suffices to consider only those tuples that cover every edge at least twice. Thus, there are at most $p$ distinct undirected edges used by these walks.  

%Let $G=G(w_1,w_2,\ldots, w_l)$ be the subgraph of $K(L,L)$ whose edges are those that are used by some $w_i.$  
For any edge $e \in \mathcal{E}(G),$ let the multiplicity of $e,$ $|e|,$ be the total number of times $e$ is traversed by the walks $w_i.$  Define the excess multiplicity $R=R(w_1,w_2, \ldots, w_l)$ to be
\[
	R = \sum_{e \in \mathcal{E}(G)} \left( |e| - 4 \right)_{+}.
\]
Using \eqref{eq:higher}, it is not hard to see that
\[
\Exp \prod_{i=1}^l \atomword{w_i} = O(L^{-p+(\tfrac 12 - \delta)R}),
\]
uniformly over all choices of words from $\prod_{i=1}^l \walkset{m(i)}.$

%Let $c$ be the number of connected components of $G.$
%Note that the moments up to any fixed order of $L^{p_j}\cdot \atomword{w_j}$ can be bounded uniformly over all words, all $1 \leq j \leq k,$ and all $L \geq 1.$ Hence, we have the estimate
%\[
%\Exp \prod_{i=1}^l \atomword{w_i} = O(L^{-p}),
%\]
%which holds uniformly over all choices of words from $\prod_{i=1}^l \walkset{m(i)}.$  Now, if any tuple covers an edge more than twice, then it follows there are $O(L^{p-1})$ vertex choices that yield isomorphic tuples, and hence the contribution is negligible.

%Thus, we may consider the case that every edge of the support is used exactly twice by the whole collection of walks.  
%% The following is not true and not needed.  Think of a long cycle pinched together in the center.  Its support has two cycles --EP
%In this case, the support of any one walk is either unicyclic or a DFS walk on a tree.  
Recall that the dependency graph $D$ with vertices $[l]$ is given by the edge set 
\[
\mathcal{E}(D) = \{(i,j)~:~\atomword{w_i} \text{ is not independent of } \atomword{w_j} \}.
\]
If in this dependency graph there is an isolated vertex $i$, then the corresponding variable $\atomword{w_i}$ is independent of all the other $\atomword{w_j},$ and hence 
\(
\Exp \prod_{i=1}^l \atomword{w_i} = 0.
\)  
%Thus, we assume there are at least $\lceil{l/2}\rceil$ many distinct edges in the dependency graph.  
Thus, we assume every vertex in the dependency graph has an incident edge.

Let $c$ be the number of connected components of the dependency graph. Walks whose indices are neighbors in the dependency graph necessarily overlap, and hence the support of a connected component of the dependency graph is a connected subgraph of $G.$  Call these subgraphs of $G$ the $D$-components of $G.$  Note that two $D$-components can in principle overlap on vertices, but they can not overlap on edges.  Let $t$ be the number of $D$-components that are trees. %that is traversed by multiple walks.  
This implies that the number of vertices in the support of $G$ is at most $|\mathcal{E}(G)| + t.$

Since there are no isolated vertices in the dependency graph $D$, each acyclic $D$-component of $G$ must contain an edge of multiplicity $4.$
Since every edge in $G$ has multiplicity at least $2,$ we have
\[
	R 
	= \sum_{e \in \mathcal{E}(G)} \left( |e| - 4 \right)_{+}
	= \sum_{e \in \mathcal{E}(G)} \biggl(  |e| - 2 - 2\cdot \one[|e|\geq 4] - \one[|e|=3] \biggr)
	\leq 2p - 2|\mathcal{E}(G)| - 2t.
	%- \one[R > 0],
\]

It follows that the total contribution to the mixed moment from tuples $(w_1,w_2,\dots, w_l)$ whose support is given by $G$ is
\begin{equation}
  \label{eq:mombound}
\begin{aligned}
	\sum_{
		\substack{
			(w_1,\dots,w_l) \\
			\text{support $G$}
		}
	}
\Exp \prod_{i=1}^l \atomword{w_i}
 &= O\left( 
 L^{|\mathcal{E}(G)| +t} L^{-p + (\tfrac{1}{2} -\delta)R}
 \right) \\
 &= O\left( 
 L^{|\mathcal{E}(G)| +t} L^{-p + (\tfrac{1}{2} -\delta)
 (2p - 2|\mathcal{E}(G)| - 2t)}
 \right) \\
 &= O\left(L^{ -2\delta(p-|\mathcal{E}(G)| - t)}\right).
\end{aligned}
\end{equation}
Let us show a lower bound for $p$ in terms of $|\mathcal{E}(G)|$ and $t.$
\[
  2p = \sum_{e \in \mathcal{E}(G)} |e| \geq 2|\mathcal{E}(G)| + 2t,
\]
where we have used that every component that is a tree contains an edge of multiplicity $4.$
If the inequality is strict, that is $p>|\mathcal{E}(G)| + t,$ the contribution to the mixed moment in \eqref{eq:mombound} is trivial.

When $p= |\mathcal{E}(G)| + t,$ every $D$-component of $G$ that contains a cycle must have all multiplicities equal to $2$ and every acyclic $D$-component of $G$ has exactly one edge of multiplicity $4.$  The only way to create such an acyclic $D$-component is for two depth first search walks to overlap on a single edge. Thus every such acyclic $D$-component has size $2.$ 

A $D$-component of $G$ whose every edge has multiplicity equal to $2$ can be composed of possibly multiple walks all of whose multiplicity-$1$ edges are matched.  However, if we let $\mathcal{C} \subset [l]$ be the indices of such a $D$-component, then
\(
\Exp \prod_{i \in \mathcal{C}} \atomword{w_i} = O(L^{-\sum_{i \in \mathcal{C}} p_i}).
\)
Note that $\sum_{i \in \mathcal{C}} p_i$ is the number of edges in the $D$-component of $G$ using these indices.  The number of vertices of this $D$-component is at most the number of edges, with equality if and only if the graph is unicyclic.  In particular, if the $D$-component is not unicyclic, the contribution of the tuples $(w_1, \dots, w_l)$ to the mixed moment will be $O(L^{-1}).$
A $D$-component is unicyclic only if $|\mathcal{C}|=2.$  

From the previous two paragraphs, all connected components of $D$ must have size $2$ to have a non-negligible contribution to the mixed moment, i.e.\,$D$ must be a perfect matching, from which follows the Wick formula.

It is instructive to note that the only difference between the situation here and the statement for a single Wishart matrix is that the number of vertex choices depends on the sizes of the submatrices chosen.  However, as all dimensions scale linearly with $L,$ the needed $O(L^{|\mathcal{E}(G)|+t})$ estimate still holds.

\end{proof}

We proceed to the proof for planar test functions, which again uses the method of moments, first establishing the covariance convergence and then proving the needed form of the moments.  The covariance calculation can be seen to follow from Lemma~\ref{lem:wish_cov_GF} together with the following elementary lemma.
\begin{lemma}
\label{lem:cov_bound}
There is a constant $C_{k,l}$ so that for all $0 \leq y,z \leq M,$ 
\[
\left|\Exp \CLS{x^k}{y}\CLS{x^l}{z}\right| \leq C_{k,l}.
\]
\end{lemma}
The convergence in the covariance limit can in fact be made uniform over $0 \leq y,z \leq M,$ and hence this lemma is an immediate consequence.  Thus, we skip its proof and turn to the extension to planar functions.

\begin{proof}[Proof of Proposition~\ref{prop:CLSplanar}]
From Lemma~\ref{lem:cov_bound}, it follows that for any polynomials $p_i$ and $p_j$ there is a constant $C$ so that 
\[
\left|\Exp \CLS{p_i}{y}\CLS{p_j}{z}\right| \leq C
\]
for all $1 \leq i,j \leq k$ and all $0\leq y,z \leq M,$ where $M$ is sufficiently large to contain the supports of all probability measures $\rho_j.$  Now we notice that
\[
\Exp \PCLS{p_i}{\rho_i}\PCLS{p_j}{\rho_j}
= 
\int_{0}^\infty
\int_{0}^\infty
\Exp
\CLS{p_i}{y}
\CLS{p_j}{z}~d\rho_i(y)d\rho_j(z),
\] 
on account of $\PCLS{p_i}{\rho_i}$ and $\PCLS{p_j}{\rho_j}$ being finite linear combinations of linear statistics.  Note that pointwise convergence of the integrand to the desired quantity follows from Lemma~\ref{lem:wish_cov_GF}.  The convergence of the integral to the desired quantity now follows by dominated convergence.

It remains to show that the variables are Gaussian, which follows in a nearly identical fashion to Proposition~\ref{prop:CLS}.  To see this, define
\(
\walksetf{j}{y}
\)
to be all closed walks on $K([\mu y L],[\nu y L])$ starting from the $[\mu y L]$ side and having length $2j.$  Now we can write
\[
\PCLS{x^{p_i}}{\rho_i} 
= \int_0^\infty
\sum_{w \in \walksetf{i}{y}} \atomword{w}~d\rho_i(y),
\]
which is again a finite sum of linear statistics.  In particular, we can write
\[
\PCLS{x^{p_i}}{\rho_i} 
=
\sum_{w \in \walksetf{i}{M}} \rhoweight{\rho_i}{w}\atomword{w},
\]
where $\rhoweight{\rho_i}{y} = \rho_i([y^*, M]),$ for $y^*$ is the smallest non-negative real so that $w$ is a walk on $K([\mu y L],[\nu y L]).$  As all these coefficients are uniformly bounded by $1,$ the remainder of the proof is now identical to the proof of~\ref{prop:CLS}. 
\end{proof}

\section{Extension argument}
\label{sec:extension}
The main purpose of this section is to prove Theorem~\ref{thm:pcls}.  We commence with the proof of Lemma~\ref{lem:rho_var}.
\begin{proof}[Proof of Lemma~\ref{lem:rho_var}]
%From Proposition~\ref{prop:shcher_var}, we have that for any $s > \tfrac32$ there is a constant $C_s$ so that
%\[
%\Var \CLS{\phi}{1} \leq C_s \FSN[\phi][s]^2.
%\]
%We would like to integrate $\Var \CLS{\phi}{y}.$

	Recall that for all $0<y<1$, $\Wmf{y} = \smf{y}^* \smf{y}/L$ with $\smf{y}$ a $[y \mu L] \times [y \nu L]$ submatrix of the $[\mu L] \times [\nu L]$ matrix $G$.  
	It is possible to construct sequences $m^{i}(L)$ for $i = 1,2,\ldots,\lfloor \tfrac1 \nu \rfloor$ with $m^{i}(L)/L \to \mu/\nu$ and $m^{i}(L) \times L$ submatrices $Q_{i,L}$ of $\ambientmatrix$ so that 
	\[
		\left\{
			\smf{y} : y > 0, L \in \mathbb{N}
		\right\}
		\subseteq
		\left\{
			Q_{i,L} :i = 1,2,\ldots,\lfloor \tfrac1 \nu \rfloor, L \in \mathbb{N} 
		\right\}.
	\]
	Hence by Proposition~\ref{prop:shcher_var}, there is a constant $C_s>0$ so that  
\[
\Var \CLS{\phi}{y} \leq C_s \FSN[\phi \circ D_{[L\nu y]/L} ][s]^2,
\]
where $D_t$ is the dilation $x \mapsto tx.$  For any $t$ in a compact set $K$, it is easily checked that there is a $C= C(K) > 0$ so that
\[
\FSN[\phi \circ D_{t} ][s]^2 
\leq C
\FSN[\phi ][s]^2. 
\]

Let $K$ be the support of $\rho.$ By the prior reasoning, we have that there is a constant $\tilde{C}_s>0$ so that
\[
\Var \CLS{\phi}{y} \leq \tilde{C}_s \FSN[\phi][s]^2
\]
for all $y \in K.$  The proof now follows from Jensen's inequality:
\begin{align*}
\Var X_{\psi \rho}
&=
\Var \int_0^{\infty}\CLSf{\psi(\cdot,y)}{y}~d\rho(y) \\
&\leq
\int_0^{\infty}
\Var \CLSf{\psi(\cdot,y)}{y}~d\rho(y) \\
&\leq
\int_0^{\infty}
\tilde{C}_s \FSN[\psi][s]^2~d\rho(y) = \tilde{C}_s \FSNr[\psi]^2. 
\end{align*}

\end{proof}

\begin{proof}[Proof of Theorem~\ref{thm:pcls}]
We wish to extend the test functions in Proposition~\ref{prop:CLSplanar} to non-polynomial test functions, in particular to measures $\psi(x,y)\rho(y)$ where $\psi$ has finite $\FSNr$ norm.  As $\rho$ is compactly supported, we may find a $K>0$ so that $[0,K]$ contains its support.  Let $\opnorm$ denote the operator norm of a finite dimensional rectangular matrix. It is well known that $\smf{K}$ has operator norm $O(L^{1/2});$ by Theorem 9.13 of~\cite{BaiBook}, there is a constant $C$ so that with probability going to $1,$ $\opnorm[\smf{K}] \leq C\sqrt{L}.$  Thus in addition, the median $\opnorm[\smf{K}]$ is at most $\sqrt{L}.$

	As $\opnorm$ is $1$-Lipschitz (with respect to the Frobenius norm on the matrix), convex, and the entries of $\smf{K}$ are supported on $[-L^{1/2-\delta},L^{1/2 - \delta}]$ we have the following consequence of Talagrand's inequality (see \cite[Theorem 4.4.10]{AnGuZe}):
	\begin{equation}
		\label{eq:tal}
		\Pr\left[
		\opnorm[ \smf{K} ] > (C+t)\sqrt{L} 
		\right]
		\leq C\exp(-L^{2\delta}t^2/C)
	\end{equation}
	where we have increased $C$ if necessary and $t \geq 0$ is arbitrary.
	Hence by interlacing of singular values (see~\cite{HJtopics}), all eigenvalues of $\{\Wmf{y}\}_{y \leq K}$ are supported on $[0,(C+t)^2]$, with failure probability at most $C\exp(-L^{2\delta}t^2/C).$

	For any $M \geq C,$ let $\tau_M : \R \to \R$ be any $C^{\infty}$ compactly supported cutoff function that is $1$ on $[-M,M].$  Then $p(x)\tau_M(x) \in \FSS.$  Moreover, as we can bound $|p(x)| \leq C'(1+x)^d$ for some $C',d >0,$  we have
	\begin{eqnarray*}
		\Exp \left| \PCLS{p\tau_M}{\rho} - \PCLS{p}{\rho} \right|
 		& \leq &
		LC' \Exp \left[ (1+ \opnorm[ \Wmf{K} ]^2)^d \one[||W(K)||_{{\mbox{op}}} \geq M 
%\opnorm[\Wmf{K}] \leq M
] 
\right] \\
		& \leq & 
		LC' \int_M^\infty d(1+t)^{d-1}
		\Pr \left[
			\opnorm[ \Wmf{K} ]^2 > t 
		\right]
		\,dt \\
		& \leq & LC''\int_M^\infty (1+t)^d \exp(-L^{2\delta}t/C'')\,dt,
	\end{eqnarray*}
	for some other constant $C'' >0.$  As $L \to \infty,$ this goes to $0.$  Therefore, the conclusion of Proposition~\ref{prop:CLSplanar} extends to measures of the form $p(x)\tau_M(x) \rho(y).$  

Similarly, the CLT can be seen to hold for test functions of the form $p(x) \tau_M(x) \one[y \leq r] \rho(y)$; write $\Psi_{p,M,r}:=p(x) \tau_M(x) \one[y \leq r]\rho(y).$ The collection $\{\Psi_{p,M,r}\}_{p,M,r}$ is easily seen to have dense span in the space $\{ \psi: \H \to \R, \FSNr[\psi] < \infty \}$, as $p$ varies over all polynomials, $M$ varies over $[C,\infty)$ and $r$ varies over $(0,\infty).$  It now follows from Lemma~\ref{lem:rho_var} and Proposition 3 of~\cite{Shcherbina11} that the desired CLT holds for $\PCLSg{\psi(x,y)\rho(y)}$ for any $\psi$ for which $\FSNr[\psi] < \infty.$

\end{proof}

\section{Acknowledgements} We would like to thank Alexei Borodin and
  Mariya Shcherbina for fruitful discussions about this paper. 
  We would also like to thank Jeffrey Kuan and Zhengye Zhou for fixing some incorrect constants in Proposition 1.2 and Lemma 2.2 in an earlier draft.

\bibliographystyle{alpha}
\bibliography{WM}

\end{document}